\documentclass[12pt]{amsart}
\usepackage{amsfonts, amssymb, latexsym, hyperref, xcolor, tikz,tikz-cd,transparent}
\usepackage{ytableau}
\usepackage{dynkin-diagrams}
\usepackage[letterpaper,left=2.25cm,right=2.25cm,top=3cm,bottom=3cm,headsep=1cm]{geometry}
\usetikzlibrary{topaths, calc, shapes,arrows,fit,positioning}
\usepackage{graphicx}
\usepackage{xparse}
\usepackage{mathtools}
\usepackage{hyperref}
\usepackage{subcaption}

\hypersetup{colorlinks, linkcolor={blue!50!purple},
citecolor={green!60!blue}, urlcolor={blue!80!black}}

\newcommand{\Z}{\mathbb{Z}}

    \newtheorem{theorem}{Theorem}[section]
    \newtheorem{lemma}[theorem]{Lemma}
    \newtheorem{prop}[theorem]{Proposition}
    
    \newtheorem*{thm*}{Theorem}
    
\theoremstyle{definition}
    \newtheorem{definition}[theorem]{Definition}
    \newtheorem{example}[theorem]{Example}
\theoremstyle{remark}
    
\numberwithin{equation}{section}

\title{Sums of Schubert structure constants with bounded Coxeter length}
\author{Ada Stelzer}
\address{Dept.~of Mathematics, U.~Illinois at Urbana-Champaign, Urbana, IL 61801, USA}
\email{astelzer@illinois.edu}
\date{16 June 2025}

\begin{document}

\maketitle

\begin{abstract}
    Pak--Robichaux recently introduced a signed puzzle rule for Schubert structure constants, which they use to show that sums $\gamma_k(n)$ of these constants with a bounded number of inversions are polynomial. We give a different, conceptual proof of their theorem. Our argument computes the lead term of $\gamma_k(n)$ and extends to all classical Lie types.
\end{abstract}

\section{Introduction}

Let $\{G_n\}_{n\geq 1}\in\{\{SL_{n+1}\}, \{SO_{2n+1}\}, \{SP_{2n}\}, \{SO_{2n}\}\}_{n\geq 1}$ be one of the families of complex classical Lie groups, each with maximal Borel subgroup $B$ and maximal torus $T$. Each \emph{generalized flag variety} $G_n/B$ has a cell decomposition with cells indexed by elements of the \emph{Weyl group} $W_n := N(T)/T$. The \emph{Schubert varieties} $\mathfrak{X}_w$ for $w\in W_n$ are the Zariski closures of these cells; via Poincar\'e duality, the Schubert varieties yield cohomology classes $\sigma_w\in H^\star(G_n/B)$. These \emph{Schubert classes} in fact form a $\Z$-linear basis for $H^\star(G_n/B)$. Thus there exist integers, the \emph{Schubert structure constants} $c_{u, v}^w$, such that
\[\sigma_u\smile \sigma_v = \sum_{w\in W_n}c_{u, v}^w\sigma_w.\]
The Schubert structure constants are known to be nonnegative for geometric reasons, and it is a major open problem to give a positive combinatorial rule for them.

For a fixed choice of classical type, consider the sum of Schubert structure constants
\[\gamma_k(n) := \sum_{\substack{u, v, w\in W_n \\ \ell(w) = k}}c_{u, v}^w,\]
where $\ell(w)$ denotes the \emph{Coxeter length} of $w$. Recently, Pak--Robichaux introduced a signed puzzle rule \cite[Theorem 1.1]{PR} for $c_{u, v}^w$ in the case where $G_n = SL_{n+1}$ is the special linear group and $W_n = \mathcal{S}_{n+1}$ is the symmetric group. They give the following result as an application:

\begin{theorem}[{\cite[Theorem 1.2]{PR}}]\label{thm:inspo}
    If $\{G_n\} = \{SL_{n+1}\}$, then $\gamma_k(n)$ is polynomial in $n$.
\end{theorem}

In \cite[Section 8.3]{PR}, the authors state that it would be interesting to find a more conceptual argument for Theorem~\ref{thm:inspo}, especially one which explicitly determines the degree (or more) of $\gamma_k(n)$. We give such an argument, proving the following strengthening of Theorem~\ref{thm:inspo}.

\begin{theorem}\label{thm:main}
   Let $\{G_n\} \in \{\{SL_{n+1}\}, \{SO_{2n+1}\}, \{SP_{2n}\}, \{SO_{2n}\}\}$. For sufficiently large $n$, $\gamma_k(n)$ is a polynomial in $n$ with lead term 
   \[LT(\gamma_k(n)) = \frac{(2n)^k}{k!}.\]
\end{theorem}

See Theorem~\ref{thm:fullform} for a precise quantification of ``sufficiently large'' in each classical type. Our argument uses the \emph{equivariant restriction} formula of Andersen--Jantzen--Soergel \cite{AJS}, and specifically the known stability properties of Schubert structure constants it implies. The main idea is that as $n$ grows, the sum defining $\gamma_k(n)$ involves many distinct elements $w\in W_n$ for which the sum $\sum_{u, v\in W_n}c_{u, v}^w$ is the same. Formalizing this equivalence relation on $W_n$ allows one to rewrite the formula for $\gamma_k(n)$ using only Weyl group elements in $W_{2k}$. This reformulation makes the polynomality of $\gamma_k(n)$ obvious, and the $c_{u, v}^w$ contributing to the lead term of this polynomial are the easiest Schubert structure constants to compute. 

\section{Background on equivariant restrictions}
We give a brief overview of the results in Lie theory and equivariant cohomology used in our proof of Theorem~\ref{thm:main}, following the exposition and notation of \cite{RYY}. See \cite{RYY} for further reading and references. The four families of classical complex Lie groups, $SL_{n+1}$, $SO_{2n+1}$, $SP_{2n}$, and $SO_{2n}$, are characterized by certain (multi)graphs, the \emph{Dynkin diagrams} of types $A_n$, $B_n$, $C_n$, and $D_n$ respectively. Each node in a Dynkin diagram corresponds to a simple reflection generating the associated Weyl group. Different types of edges represent commutation relations between these generators, with non-adjacent nodes corresponding to commuting pairs of generators. Examples of the classical Dynkin diagrams for $n = 5$ are below.

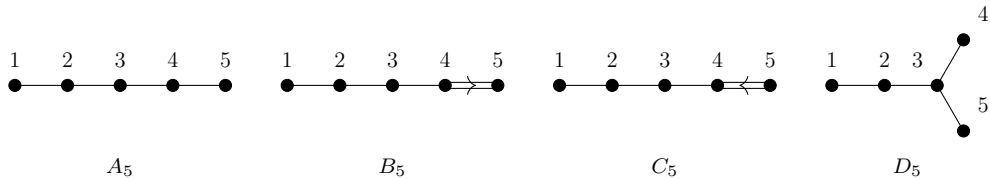
\[\underbracket[0pt]{\vphantom{\dynkin[edge length = 0.7]D{}}\dynkin[label*, root radius=.08cm, edge length = 0.7]A5}_{A_5} \quad
\underbracket[0pt]{\vphantom{\dynkin[edge length = 0.7]D{}}\dynkin[label*, root radius=.08cm, edge length = 0.7]B5}_{B_5} \quad
\underbracket[0pt]{\vphantom{\dynkin[edge length = 0.7]D{}}\dynkin[label*, root radius=.08cm, edge length = 0.7]C5}_{C_5} \quad
\underbracket[0pt]{\dynkin[label*, root radius=.08cm, edge length = 0.7]D5}_{D_5}\]

For our purposes, we only need to know that Dynkin diagrams encode the generators and relations of Weyl groups. In particular, call a map $\iota:\Delta\to\Delta'$ of Dynkin diagrams an \emph{embedding} if it realizes $\Delta$ as an induced subgraph of $\Delta'$. Then $\iota$ induces an embedding $W\hookrightarrow W'$ of the corresponding Weyl groups. For example, the fact that $B_n$, $C_n$, and $D_n$ contain embedded copies of $A_{n-1}$ indicates that the associated Weyl groups contain symmetric groups as subgroups.

Now, let $G_n/B$ be a generalized flag variety with Weyl group $W_n$ and simple roots $\{\alpha_1,\dots, \alpha_n\}$. There is a torus $T$ acting on $G_n/B$ on the left, which stabilizes each Schubert variety $\mathfrak{X}_w$ ($w\in W_n$). Thus each $\mathfrak{X}_w$ gives rise to an \emph{equivariant Schubert class} $\xi_w$ in the $T$-equivariant cohomology $H^\star_T(G_n/B)$. The equivariant Schubert classes form a basis for $H^\star_T(G_n/B)$ as a $H^\star_T(pt)$-module, where we identify $H^\star_T(pt)$ with the polynomial ring $\Z[\alpha_1,\dots, \alpha_n]$. Thus we obtain \emph{equivariant Schubert structure constants} $C_{u, v}^w\in\Z[\alpha_1,\dots, \alpha_n]$, defined by the equation
\[\xi_u\smile\xi_v = \sum_{w\in W_n}C_{u, v}^w\xi_w.\]
Each $C_{u, v}^w$ is a homogeneous polynomial of degree $\ell(w)-\ell(u)-\ell(v)$. Setting $\alpha_i = 0$ for all $i$ recovers ordinary cohomology; in particular, $C_{u, v}^w = c_{u, v}^w$ whenever $\ell(w) = \ell(u)+\ell(v)$ (i.e., whenever $c_{u, v}^w\neq 0$). The advantage in working with $H^\star_T(G_n/B)$ comes from \emph{GKM theory} \cite{GKM}: the set $(G_n/B)^T$ of $T$-fixed points of $G_n/B$ is indexed by $W_n$ (hence finite), and it turns out that the map
\begin{equation}\label{eqn:GKM}
    H^\star_T(G_n/B)\to H^\star_T((G_n/B)^T)\cong\bigoplus_{w\in W_n}\Z[\alpha_1,\dots, \alpha_n]
\end{equation}
on cohomology induced by the inclusion $(G_n/B)^T\hookrightarrow G_n/B$ is also an injection. Thus the product structure on $H^\star_T(G_n/B)$ is given by pointwise products of lists of polynomials.
\begin{definition}
    The image of $\xi_v$ under the embedding (\ref{eqn:GKM}) is denoted $(\xi_{v|w})_{w\in W_n}$. The polynomial $\xi_{v|w}\in\Z[\alpha_1,\dots, \alpha_n]$ is called the \emph{equivariant restriction} of $\xi_v$ at $w$.
\end{definition}
Andersen--Jantzen--Soergel gave a formula for $\xi_{v|w}$ involving the combinatorics of $W_n$ and its associated root system. Let $s_{\alpha_i}$ denote the simple reflection through the hyperplane orthogonal to $\alpha_i$. Let $I = s_{\beta_1}\dots s_{\beta_{\ell(w)}}$ be a reduced word for $w\in W_n$. For each subword $J\subseteq I$ and $i\in[\ell(w)]$, let
\[\gamma_{J, i} := \begin{cases}
    \beta_i s_{\beta_i} & \beta_i\in J,\\
    s_{\beta_i} & \beta_i\notin J.
\end{cases}\]

\begin{theorem}[Equivariant restriction, \cite{AJS}]\label{thm:eqres}
    For $v$, $w$, $I$, $J$, and $\gamma_{J, i}$ as above, the equivariant restriction of $\xi_v$ at $w$ is the polynomial
    \[\xi_{v|w} = \sum_{J\subseteq I}\prod_{i=1}^{\ell(w)}\gamma_{J, i}\cdot 1\in\Z[\alpha_1,\dots, \alpha_n],\]
    where the sum is over subwords $J\subseteq I$ that are reduced words for $v$.
\end{theorem}

\begin{example}
    Let $w = 3241\in\mathcal{S}_4$ and choose the reduced word $I = s_1s_2s_3s_1$ for $w$. Let $v = 2143$, which has reduced words $s_1s_3$ and $s_3s_1$. Using the type-$A$ formula
    \[s_{\alpha_i}\cdot \alpha_j = \begin{cases} -\alpha_i & i = j,\\
        \alpha_i+\alpha_j & |i-j| = 1,\\
        \alpha_j & |i-j|\geq 2,\end{cases}\]
    we compute $\xi_{v|w}$ using Theorem~\ref{thm:eqres}:
    \[\xi_{v|w} = \alpha_1 s_1 s_2 \alpha_3 s_3 s_1\cdot 1 + s_1 s_2 \alpha_3 s_3 \alpha_1 s_1\cdot 1 = \alpha_1 s_1s_2\alpha_3 + s_1s_2\alpha_3\alpha_1 = (\alpha_1+\alpha_2)(\alpha_1+\alpha_2+\alpha_3).\]
\end{example}

The restriction formula gives a slick, type-independent method for proving facts about Schubert structure constants without identifying explicit representatives for the cohomology classes. We record some well-known consequences we use in the proof of Theorem~\ref{thm:main}.

\begin{lemma}\cite[pg. 2]{RYY}\label{lemma:0restriction}
    The equivariant restriction $\xi_{v|w} = 0$ unless $v\leq w$ in Bruhat order. The equivariant structure constant $C_{u, v}^w = 0$ unless $u\leq w$ and $v\leq w$ in Bruhat order.
\end{lemma}

\begin{prop}\cite[Theorem 2.1]{RYY}\label{prop:equiembed}
    If $\iota:\Delta\to \Delta'$ is an embedding of Dynkin diagrams, then $C_{u, v}^w$ and $C_{\iota(u), \iota(v)}^{\iota(w)}$ agree up to labelling of the variables. In particular, $c_{u, v}^w = c_{\iota(u), \iota(v)}^{\iota(w)}$.
\end{prop}

\begin{prop}\label{prop:equiprod}
    If $w = w'\times w''$ factors over a parabolic subgroup $W'\times W''\leq W$, then $C_{u, v}^w = 0$ unless $u = u'\times u''$ and $v = v'\times v''$ also factor. In this case,
    \[C_{u, v}^w = C_{u', v'}^{w'}C_{u'', v''}^{w''}.\]
\end{prop}
\begin{proof}
    The first sentence is immediate from Lemma~\ref{lemma:0restriction} and the fact that Bruhat order respects factorizations. Now, if $I$ is a reduced word for $w$, then using the commutation relations for Weyl groups we may choose $I$ to be of the form $I'I''$, where $I'$ is a reduced word for $w'$ and $I''$ is a reduced word for $w''$. We claim first that for $v = v'\times v''\leq w$, we have
    \[\xi_{v|w} = \xi_{v'|w'}\cdot \xi_{v''|w''}.\]
    Note that by our choice of $I = I'I''$, each $J$ appearing in the formula for $\xi_{v|w}$ must be of the form $J'J''$, where $J'$ is a reduced word for $v'$ and $J''$ is a reduced word for $v''$. Moreover, the simple reflections $s_{\beta_{i'}}$ associated to $\beta_{i'}\in I'$ fix all the roots $\beta_{i''}\in J''$, since these roots are orthogonal in the root system for $W'\times W''$. Our claim follows:
    \[\xi_{v|w} = \sum_{J\subseteq I'I''}\prod_{i=1}^{\ell(w)}\gamma_{J, i}\cdot 1 = \left(\sum_{J'\subseteq I'}\prod_{i'=1}^{\ell(w')}\gamma_{J', i'}\cdot 1\right)\left(\sum_{J''\subseteq I''}\prod_{i''=1}^{\ell(w'')}\gamma_{J'', i''}\cdot 1\right) = \xi_{v'|w'}\xi_{v''|w''}.\]
    Now, let $x$ denote the longest element in $W'\times W''$. Then by definition
    \[\xi_{u|x}\cdot \xi_{v|x} = \sum_{w\in W}C_{u, v}^w\xi_{w|x}.\]
    By Lemma~\ref{lemma:0restriction}, the only nonzero terms in the sum come from $w$ such that $u, v\leq w\leq x$, and all such $w$ factor as $w'\times w''\in W'\times W''$. But we also have 
    \[\xi_{u|x}\xi_{v|x} = \xi_{u'|x'}\xi_{u''|x''}\xi_{v'|x'}\xi_{v''|x''} = \sum_{w'\in W', w''\in W''}C_{u', v'}^{w'}\xi_{w'|x'}C_{u'', v''}^{w''}\xi_{w''|x''} = \sum_{w = w'\times w''}C_{u', v'}^{w'}C_{u'', v''}^{w''}\xi_{w|x}.\]
    The proposition follows by the uniqueness of equivariant Schubert structure constants.
\end{proof}

\section{Proof of Theorem~\ref{thm:main}}
Throughout this section, we will fix $\mathcal{W} := \{W_n\}_{n\geq 1}$ to be the sequence of Weyl groups associated to one of the classical familes of complex Lie groups. The associated Dynkin diagrams will be denoted $\Delta_n$.

\begin{definition}
    Let $I$ be a reduced word for $w\in W_n$. The \emph{Dynkin support} of $w$ is the Dynkin diagram $\Delta(w)$ isomorphic to the subgraph of $W_n$ induced by the vertices corresponding to those $s_{\alpha_j}$ appearing in $I$. 
\end{definition}

\begin{definition}
    We say two elements $w\in W_m$ and $w'\in W_{n}$ are \emph{equivalent} and write $w\sim w'$ if there is an isomorphism of Dynkin supports $\Delta(w)\cong\Delta(w')$ sending $w$ to $w'$.
\end{definition}

Since all reduced words for $w$ involve the same set of simple reflections, $\Delta(w)$ depends only on $w$ and not on the choice of $I$. Coxeter length and Dynkin support are invariant under $\sim$, so we may refer to $\ell([w])$ or $\Delta([w])$ for an equivalence class $[w]$. Since $\Delta(w)$ is (isomorphic to) an induced subgraph of $\Delta_n$, we see that $\Delta(w)$ has at most one component of type $B$, $C$, or $D$ (when $\Delta_n$ is also of this type) and all other components are of type $A$. Writing $w = \prod_{i=1}^p w_i$ as an element of the Weyl group associated to $\Delta(w)$, for $1\leq j\leq n$ let $a_w(j)$ denote the number of factors $w_i$ in components of type $A_j$ that are equal.

\begin{lemma}\label{lemma:finite}
    If $w\in W_n$ has Coxeter length $k$, then there exists $w'\in W_{2k}$ such that $w\sim w'$.
\end{lemma}
\begin{proof}
    Since $w$ has length $k$, a reduced word $I$ for $w$ uses at most $k$ distinct simple reflections. Thus $\Delta(w)$ has at most $k$ vertices, and any such diagram embeds as an induced subgraph of any classical Dynkin diagram with $2k$ vertices.
\end{proof}

\begin{lemma}\label{lemma:enum}
    Let $\Delta$ be a Dynkin diagram with $k'$ vertices and $c$ components, all of type $A$. Then for all $k\geq k'-1$, the number of embeddings $\Delta\hookrightarrow A_k$ is
    \[\frac{(k+1-k')!}{(k+1-k'-c)!} = \prod_{i=0}^{c-1}(k+1-k'-i),\]
    a polynomial in $k$ with lead term $k^c$.
\end{lemma}
\begin{proof}
    An embedding $\iota:\Delta\hookrightarrow A_k$ is equivalent to an ordering of the $c$ (labelled) factors of $\Delta$ and the $(k-k')$ (unlabelled) vertices of $A_k\setminus\iota(\Delta)$ in which no two components of $\Delta$ appear consecutively (see Figure~\ref{fig:embedding}). Such orderings are equinumerous with orderings of $c$ labelled objects and $(k+1-k'-c)$ unlabelled objects (with no adjacency restrictions), which are enumerated by the formula in the lemma statement.
\end{proof}

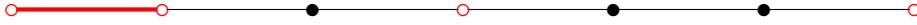
\begin{figure}[h]
    \begin{center}
        \begin{tikzpicture}
            \draw[red, ultra thick] (0.06,0.02) -- (1.94,0.02);
            \draw[black, thin] (2.06,0.02) -- (3.94,0.02);
            \draw[black, thin] (4.06,0.02) -- (5.94,0.02);
            \draw[black, thin] (6.06,0.02) -- (7.94,0.02);
            \draw[black, thin] (8.06,0.02) -- (9.94,0.02);
            \draw[black, thin] (10.06,0.02) -- (11.94,0.02);
        
            \draw[red] (0,0) node{$\circ$};
            \draw[red] (2,0) node{$\circ$};
            \draw (4,0) node{$\bullet$};
            \draw[red] (6,0) node{$\circ$};
            \draw (8,0) node{$\bullet$};
            \draw (10,0) node{$\bullet$};
            \draw[red] (12, 0) node{$\circ$};
        \end{tikzpicture}
    \end{center}
    \caption{An embedding of $\Delta = A_1\times A_1 \times A_2$ into $A_7$.}\label{fig:embedding}
\end{figure}

\begin{prop}\label{prop:poly}
    For $w\in W_m$ for some $m$, let $N_w(n)$ denote the number of elements $v\in W_n$ such that $w\sim v$. Let $n'$ denote the number of vertices in $\Delta(w)$, and $p$ the number of type-$A$ components in $\Delta(w)$. Then for
    \[n\geq \begin{cases}
        n'-1 & \text{in type $A$},\\
        n' & \text{in types $B$ and $C$},\\
        n'+1 & \text{in type $D$},
    \end{cases}\]
    $N_w(n)$ is a polynomial with lead term 
    \[LT(N_w(n)) = \frac{n^p}{\prod_{j=1}^m (a_w(j)!)}.\] 
\end{prop}
\begin{proof}
    Let $N'_w(n)$ denote the number of induced subgraphs of $\Delta_n$ isomorphic to $\Delta(w)$. Then
    \[N_w(n) = \frac{N'_w(n)}{\prod_{j=1}^m (a_w(j)!)},\]
    so it suffices to show that $N'_w(n)$ is a polynomial with lead term $n^p$ for sufficiently large $n$. In type $A$, this follows immediately by Lemma~\ref{lemma:enum}.
    
    In types $B$, $C$, and $D$, the analysis is similar. If $\Delta(w)$ contains a component that is not of type $A$ (say with $r$ vertices), then this component is unique and has a unique embedding into each $W_n$. The remaining components of $\Delta(w)$ are then embedded into a copy of $A_{n-r-1}$. Applying Lemma~\ref{lemma:enum} with $k = n-r-1$, $k' = n'-r$, and $c = p$ shows that these embeddings are enumerated by a polynomial with lead term $n^p$ for $n\geq n'$.
    
    The enumeration of embeddings is slightly more complicated when all components of $\Delta(w)$ are of type $A$. Consider first the Dynkin diagrams of type $B_n$ or $C_n$, which each contain one copy of $A_{n-1}$. By Lemma~\ref{lemma:enum} (with $k = n-1$, $k' = n'$, $c = p$), embeddings of $\Delta(w)$ into this embedded type-$A$ subdiagram are enumerated by a polynomial with lead term $n^p$ for $n\geq n'$. There are no further embeddings $\Delta(w)\hookrightarrow \Delta_n$ unless $\Delta(w)$ contains a component isomorphic to $A_1$. In this case, additional embeddings are obtained by mapping a type-$A_1$ component of $\Delta(w)$ to the final vertex of $\Delta_n$, then embedding the rest of $\Delta(w)$ into a copy of $A_{n-2}\hookrightarrow\Delta_n$. By Lemma~\ref{lemma:enum} (with $k = n-2$, $k' = n'-1$, $c = p-1$), such ``exceptional'' embeddings are enumerated by a polynomial of degree $p-1$ for $n\geq n'$. Thus in types $B$ and $C$, $N'_w(n)$ is a polynomial with lead term $n^p$ for all $n\geq n'$ as desired.

    Consider now the case where $\Delta_n = D_n$. The Dynkin diagram $D_n$ contains two embedded copies of $A_{n-1}$, which intersect in a copy of $A_{n-2}$. By inclusion-exclusion and three applications of Lemma~\ref{lemma:enum} (twice with $k = n-1$, $k' = n'$, $c = p$ and once with $k = n-2$, $k' = n'$, $c = p$), for $n\geq n'+1$ the number of embeddings of $\Delta(w)$ into one of these type-$A$ subdiagrams is enumerated by
    \[2\cdot\frac{(n-n')!}{(n-n'-p)!}-\frac{(n-1-n')}{(n-1-n'-p)},\]
    which is a polynomial with lead term $n^p$. Moreover, all embeddings $\Delta(w)\hookrightarrow D_n$ are embeddings into these type-$A$ subdiagrams unless $\Delta(w)$ contains at least two $A_1$ components or at least one $A_3$ component (see Figure~\ref{fig:exceptional} for an example). As in the previous case, Lemma~\ref{lemma:enum} (with $(k, k', c) = (n-3, n'-2, p-2)$ and $(n-4, n'-3, p-1)$ respectively) shows that for $n\geq n'$ these ``exceptional'' embeddings are enumerated by polynomials of degree strictly less than $p$. Thus the lead term of $N'_w(n)$ is $n^p$ in all cases, completing the proof.
\end{proof}

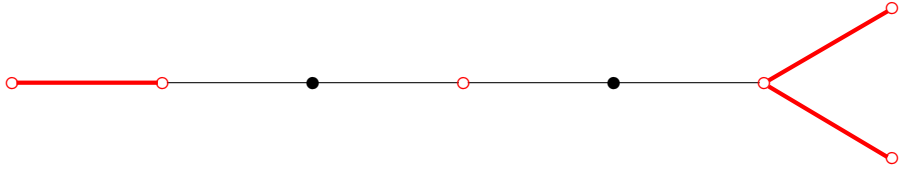
\begin{figure}[h]
    \begin{center}
        \begin{tikzpicture}
            \draw[red, ultra thick] (0.06,0.02) -- (1.94,0.02);
            \draw[black, thin] (2.06,0.02) -- (3.94,0.02);
            \draw[black, thin] (4.06,0.02) -- (5.94,0.02);
            \draw[black, thin] (6.06,0.02) -- (7.94,0.02);
            \draw[black, thin] (8.06,0.02) -- (9.94,0.02);
            \draw[red, ultra thick] (10.06,0.05) -- (11.64,0.97);
            \draw[red, ultra thick] (10.06,-0.03) -- (11.64,-0.97);
        
            \draw[red] (0,0) node{$\circ$};
            \draw[red] (2,0) node{$\circ$};
            \draw (4,0) node{$\bullet$};
            \draw[red] (6,0) node{$\circ$};
            \draw (8,0) node{$\bullet$};
            \draw[red] (10,0) node{$\circ$};
            \draw[red] (11.7,1) node{$\circ$};
            \draw[red] (11.7,-1) node{$\circ$};
        \end{tikzpicture}
    \end{center}
    \caption{An ``exceptional'' embedding of $\Delta = A_1\times A_2\times A_3$ into $D_8$.}\label{fig:exceptional}
\end{figure}

\begin{theorem}\label{thm:fullform}
    Fix $k$ and a choice of classical Lie type. Then $\gamma_k(n)$ is the finite sum
    \[\gamma_k(n) = \sum_{\substack{[w]\\ w\in W_{2k}, \ell(w) = k}}\left(\sum_{u, v\leq w}c_{u, v}^w\right)N_w(n).\]
    In particular, $\gamma_k(n)$ is a polynomial for
    \[n\geq \begin{cases} k-1 & \text{in type $A$},\\
        k & \text{in types $B$ and $C$},\\
        k+1 & \text{in type $D$}.\end{cases}\]
\end{theorem}
\begin{proof}
    Recall that by definition,
    \[\gamma_k(n) = \sum_{\substack{u, v, w\in W_n\\ \ell(w) = k}}c_{u, v}^w.\]
    By Lemma~\ref{lemma:0restriction}, we may restrict the sum to $u, v\leq w$. By Proposition~\ref{prop:equiembed}, if $w\sim w'$ in $W_n$ then $\sum_{u, v\leq w}c_{u, v}^w = \sum_{u', v'\leq w'}c_{u', v'}^{w'}$. Thus to compute $\gamma_k(n)$ we may group together structure constants indexed by elements in the same equivalence class $[w]$, and the size of each set $[w]\cap W_n$ is precisely $N_w(n)$ by definition. Finally, Lemma~\ref{lemma:finite} shows that each equivalence class $[v]$ for $v\in W_n$ with $\ell(v) = k$ is represented by an element $w\in W_{2k}$. The polynomality of $\gamma_k(n)$ now follows from the polynomality of $N_w(n)$ established in Proposition~\ref{prop:poly}.
\end{proof}

Theorem~\ref{thm:fullform} separates the part of $\gamma_k(n)$ dependent on $n$ from the part that requires computation of Schubert structure constants. With this result, we can now prove Theorem~\ref{thm:main} after computing only the simplest $c_{u, v}^w$.

\begin{lemma}\label{lemma:schubcomp}
    If $w\in W_n$ has Dynkin support $\Delta(w) = A_1^k := \underbrace{A_1\sqcup\dots\sqcup A_1}_{\text{k factors}}$, then 
    \[\sum_{u, v\in W_n} c_{u, v}^w = 2^k.\]
\end{lemma}
\begin{proof}
    In the case where $k = 1$, we compute directly: $c_{s_1, id}^{s_1} = c_{id, s_1}^{s_1} = 1$ and $c_{u, v}^{s_1} = 0$ otherwise, so the sum is $2$ as claimed. For general $k$, if $\Delta(w) = A_1^k$, then $w$ factors as $s_1\times\dots\times s_1$. Thus by Proposition~\ref{prop:equiembed} and Proposition~\ref{prop:equiprod} we have
    \[\sum_{u, v\in W_n}c_{u, v}^w = \left(\sum_{u, v\in \mathcal{S}_2}c_{u, v}^{s_1}\right)^k = 2^k.\]
\end{proof}

\begin{proof}[Proof of Theorem~\ref{thm:main}]
    By Theorem~\ref{thm:fullform}, 
    \[LT(\gamma_k(n)) = \sum_{[w]}\left(\sum_{u, v\leq w}c_{u, v}^w\right)N_w(n),\]
    where the sum is over equivalence classes $[w]$ of length $k$ maximizing the degree of $N_w(n)$. By Proposition~\ref{prop:poly}, the degree of $N_w(n)$ is the number of components in $\Delta(w)$. Among $w\in W_{2k}$ of length $k$, the number of components of $\Delta(w)$ is maximized when $\Delta(w) = A_1^k$, i.e., when $w$ is a product of $k$ mutually orthogonal simple reflections. By Lemma~\ref{lemma:schubcomp} and Proposition~\ref{prop:poly} it follows that
    \[LT(\gamma_k(n)) = 2^k LT(N_w(n)) = \frac{(2n)^k}{k!},\]
    completing the proof.
\end{proof}

\section*{Acknowledgements}
We thank Andrew Hardt, Colleen Robichaux, and Alexander Yong for helpful conversations. We were supported by an NSF graduate fellowship and an
NSF RTG in Combinatorics (DMS 1937241) while preparing this material.

\end{document}